\documentclass[10pt]{article}
\usepackage{amssymb,amsmath,textcomp,amsfonts,amsthm}
\newtheorem{satz}{Theorem}[section]
\newtheorem{lemma}{Lemma}[section]

\newtheorem{bemerk1}{Remark}[section]

\newcommand{\iR}{\mathbb{R}}
\newcommand{\iN}{\mathbb{N}}

\newcommand{\oH}{\hspace*{0.39em}\raisebox{0.6ex}{\textdegree}\hspace{-0.72em}H}

\begin{document}
\begin{center}
{\bf\Large Bounded weak solutions of time-fractional porous medium type and more general nonlinear and degenerate evolutionary integro-differential equations}
\end{center}
\vspace{0.5em}
\begin{center}
Petra Wittbold, Patryk Wolejko and Rico Zacher
\end{center}
\begin{abstract}
We prove existence of a bounded weak solution to a degenerate quasilinear subdiffusion problem with bounded 
measurable coefficients that may explicitly depend on time. The kernel in the involved integro-differential operator w.r.t.\ time
belongs to the large class of  ${\cal PC}$ kernels. In particular, the case of a fractional time derivative of order less than
$1$ is included. A key ingredient in the proof is a new compactness criterion of Aubin-Lions type which involves function
spaces defined in terms of the integro-differential operator in time. Boundedness of the solution is obtained by the
De Giorgi iteration technique. Sufficiently regular solutions are shown to be unique by means of an $L_1$-contraction
estimate.
\end{abstract}
\vspace{0.7em}
\begin{center}
\end{center}

\section{Introduction} \label{IntroSec}
Let $T>0$ and $\Omega$ be a bounded domain in $\iR^d$. We consider the problem

\begin{equation} \label{mainproblem}
\left\{
\begin{array}{r@{\;=\;}l@{\;}l}
\partial_t\big(k\ast[u-u_0]\big)-\mbox{div}\big(A(t,x)\nabla \varphi(u)\big) & f,\quad & t\in (0,T),\,x\in\Omega\\
u & 0,\quad & t\in (0,T),\,x\in \partial\Omega\\
u|_{t=0} & u_0,\; & x\in \Omega.
\end{array}
\right.
\end{equation}
The kernel $k\in L_{1,\,loc}(\iR_+)$ is given, and $k\ast v$ denotes the
convolution on the positive halfline $\iR_+:=[0,\infty)$ w.r.t.\ the time variable,
that is
$(k\ast v)(t)=\int_0^t k(t-\tau)v(\tau)\,d\tau$, $t\ge 0$. Note that for sufficiently smooth $u$ with $u(0)=u_0$, 
\begin{equation} \label{operatorreform}
\partial_t \big(k\ast [u-u_0]\big)=k \ast \partial_t u.
\end{equation}

The kernel $k$ belongs to a large class of kernels. We assume that it satisfies the condition
\begin{itemize}
\item [{\bf ($\mathcal{PC}$)}] $k\in L_{1,\,loc}(\iR_+)$ is nonnegative and nonincreasing, and there exists a kernel $l\in L_{1,\,loc}(\iR_+)$ such that
$k\ast l=1$ on $(0,\infty)$.
\end{itemize}
In this case we say that $k$ is a kernel of type $\mathcal{PC}$ (cf. \cite{ZWH}) and also write $(k,l)\in \mathcal{PC}$. Note that $(k,l)\in {\cal PC}$ implies that $l$ is completely
positive, cf.\ \cite[Theorem 2.2]{CN} and \cite{CP1}, in particular $l$ is nonnegative.
Kernels of type $\mathcal{PC}$ have been studied by many authors, see e.g.\ 
\cite{KTT,KSVZ,Koch11,SaCa,VZ,VZ2,Za,ZWH}. Convolving the integro-differential equation in 
\eqref{mainproblem} with the kernel $l$ leads to a Volterra integral equation with a completely positive kernel.
Such equations (mainly in an abstract framework, but also in concrete cases involving partial differential operators) were studied earlier e.g.\ in \cite{CNa,CN,CP1,Grip1,JanI}.

Condition ($\mathcal{PC}$) covers a variety of the relevant integro-differential operators w.r.t.\ time that
appear in applications in physics in the context of {\em subdiffusion} processes. An important example is given by $(k,l)=(g_{1-\alpha},g_\alpha)$ with $\alpha\in(0,1)$, where
$g_\beta$ denotes the standard (or Riemann-Liouville) kernel
\[
g_\beta(t)=\,\frac{t^{\beta-1}}{\Gamma(\beta)}\,,\quad
t>0,\quad\beta>0.
\]
In this case, the term $\partial_t(k\ast v)$ becomes the classical Riemann-Liouville fractional derivative
$\partial_t^\alpha v$ of order $\alpha$, and $k \ast \partial_t v={}^c D_t^\alpha v$, the Caputo fractional derivative (cf.\
the right-hand side in (\ref{operatorreform})), if the function $v$ is sufficiently smooth, see e.g.\ \cite{KST}.
Further examples of kernels $(k,l)\in {\cal PC}$ are given by (see \cite[Section 6]{VZ})
\begin{itemize}
\item[(i)]
$k(t)=g_{1-\alpha}(t)e^{-\gamma t}$, $l(t)=g_{\alpha}(t)e^{-\gamma
t}+\gamma(1\ast[g_{\alpha}e^{-\gamma\cdot}])(t)$, $t>0$,
with $\alpha\in (0,1)$ and $\gamma>0$;
\item[(ii)]
$k(t)=\int_0^1 g_\beta(t)\,d\beta$, $l(t)=\int_0^\infty \,\frac{e^{-st}}{1+s}\,{ds}$, $t>0$.
\end{itemize}
Note that in both examples we also have $(l,k)\in {\cal PC}$ (cf.\ \cite{VZ}). Further examples can be found in
\cite{SaCa,VZ}.

Subdiffusion is an important special case of anomalous diffusive behaviour. Time-fractional (linear) diffusion equations 
with time order $\alpha\in (0,1)$ can
be used to model diffusion processes with a mean square displacement equal to a constant times $t^\alpha$ for $t>0$.
The mean square displacement is a time-dependent quantity that measures the dispersion of random processes and describes how fast particles diffuse. In the classical case of Brownian motion it is linear in $t$, that is, particles diffuse faster (for large times) than in the time-fractional case, cf.\ \cite{Hilfer1,Metz,Uch}. A power-law 
mean displacement has been observed experimentally in an abundance of processes
(e.g.\ diffusion on fractals
like e.g.\ some amorphous semiconductors),
see \cite{BG,Metz,Metz2,Uch} and the references given therein. Certain kernels of type $\mathcal{PC}$ (e.g.\ $k$ given in (ii) above) lead
to a mean square displacement that grows logarithmically for large times, a behaviour observed in so-called 
{\em ultraslow diffusion}, see \cite{DrKl,SchSokBl,Sinai} for examples from physics
and \cite{KSVZ,Koch08,Koch11,KuRy} for more mathematical background.

Another context where equations of the form \eqref{mainproblem} appear (in particular the nonlinear case)
is the modelling of dynamic processes in materials with {\em memory}. Important examples are given by
heat conduction with memory (\cite{Nun,JanI}) and the diffusion in fluids in porous media with memory (\cite{CapuFlow,JakuDiss}).

A special feature of problem \eqref{mainproblem} is that we allow for additional coefficients  $A=(a_{ij})$ which may 
also {\em explicitly depend on time}. Concerning these coefficients we merely assume
that they are measurable and bounded and that they satisfy a uniform parabolicity condition. 

As to the nonlinearity, in our main existence result, Theorem \ref{degensol}, we assume that $\varphi$ is subject to the following
conditions.
\begin{itemize}
\item [{\bf (H'$\varphi$)}] $\varphi\in C^1(\iR)$, $\varphi'(r)\ge 0$ for all $r\in \iR$, $\varphi(0)=0$, $\varphi$ is strictly increasing in a neighbourhood of $0$, and there exist $\mu, R>0$ such that
\[
0<\mu\le \varphi'(r)\quad \mbox{for all}\;r\in \iR\;\mbox{with}\;|r|\ge R.
\]
\end{itemize}
Thus $\varphi$ can {\em degenerate} in the sense that $\varphi'$ is not bounded away
from zero on all of $\iR$. An important example is given by $\varphi(r)=|r|^{m-1}r$ with $m>1$, which leads to
an equation of porous medium type.

Time-fractional diffusion equations of porous medium type (including problems with nonlocal diffusion term) have recently been studied by many authors and different methods, see e.g.\ \cite{Aka,ACV2,DiValVes,DNA19,LRS18,Plo15,VZ}. 
Abstract nonlocal in time equations in Hilbert spaces with a kernel $k$ of type $\mathcal{PC}$ and a nonlinearity in
subdifferential form are investigated in \cite{Aka}. As an application of the abstract results, an existence and uniqueness result for porous medium type problems is obtained. Abstract (stochastic) nonlinear time-fractional problems with explicit
time dependence in the nonlinearity are studied in \cite{LRS18} by means of monotonicity methods, see also the very
recent work \cite{LRS19}, where a more general kernel is considered. As a special case, the authors are able to show existence and uniqueness for deterministic and stochastic time-fractional porous medium (TFPM) equations. However, additional $L_\infty$-coefficients in the (main) nonlinear term are not covered; even in the
purely space dependent case $A=A(x)$. 
For a rather general class of kernels $k$ (including the standard kernel), in \cite{JakWitt,Sap,SchWitt}, existence and uniqueness for doubly nonlinear equations with an integro-differential operator as in 
\eqref{mainproblem} are studied in the framework of entropy solutions using the theory of accretive operators. This approach is based on results for abstract nonlinear Volterra equations \cite{Grip2,Grip1} (see also \cite{CNa,CN}) and works
as well for nonlinear terms with explicit space dependence. However, 
the case of explicit time-dependence of the nonlinearity cannot be treated in this way. 
We further mention that decay estimates for TFPM equations are derived in
\cite{DiValVes,VZ}. Numerical schemes for TFPM equations are the subject of \cite{Plo15,Plo14}.
Concerning linear and nonlinear time-fractional problems we further refer to \cite{ACV1,CLS,CP,KTT,Kim,TY,VZ2,
ZaG,ZEQ}.

The present paper contains three main results. The first one is a new compactness criterion of Aubin-Lions type
involving function spaces defined in terms of the integro-differential operator $\partial_t(k\ast \cdot)$ with an {\em arbitrary}
kernel $k$ of type $\mathcal{PC}$, see Theorem \ref{comWVH}.
In the time-fractional case, compactness results of this type were proved recently in \cite{LL18}.
Besides its high importance with regard to weak solutions to non-local in time PDEs, the fact that our compactness
result does not require higher time regularity in terms of fractional Sobolev spaces is very interesting in its
own right. 

As our second main result, Theorem \ref{degensol}, we prove existence of bounded weak $L_2$-solutions to \eqref{mainproblem}
in the degenerate case. Concerning the nonlinearity $\varphi$ we only assume condition (H'$\varphi$).
The kernel $k$ is assumed to be of type $\mathcal{PC}$ and such that the corresponding kernel $l\in L_p((0,T))$
for some $p>1$. The latter condition is crucial to obtain boundedness of the weak solution. To achieve this,
we also assume that $u_0\in L_\infty(\Omega)$ and $f\in L_{q_1}((0,T);L_{q_2}(\Omega))$ with
$q_1$ and $q_2$ sufficiently large, cf.\ condition (H'd) below.

The third main novelty is an $L_1$-contraction estimate for suitably defined weak 
solutions of \eqref{mainproblem}, see Theorem \ref{L1ConTheorem}.
Here the notion of solution differs from the one used in Theorem \ref{degensol} inasmuch as 
concerning time regularity we require now that $k\ast (u-u_0)\in {}_0H^1_1((0,T);L_1(\Omega))$,
instead of $k\ast (u-u_0)\in {}_0H^{1}_{2} ((0,T);H^{-1}(\Omega))$ (the zero means vanishing at $t=0$).
An important consequence of the contraction estimate is a uniqueness result for weak solutions $u$ in the sense
of Theorem \ref{degensol} assuming in addition that $k\ast (u-u_0)\in {}_0H^1_1((0,T);L_1(\Omega))$.
It seems to be a difficult task to derive conditions which are sufficient for this extra regularity.
Possibly, a certain time regularity for the coefficient matrix $A(t,x)$ is required. 

The paper is organized as follows. In Section \ref{Prelim} we discuss the Yosida approximation of the operator
$\partial_t(k\ast \cdot)$, which leads to the family of regularized kernels $k_\gamma$ ($\gamma>0$) of $k$.
We also state an important convexity inequality for operators of the form $\partial_t(k\ast \cdot)$.
Section \ref{CompactnessSec} is devoted to the compactness criterion of Aubin-Lions type involving an operator $\partial_t(k\ast \cdot)$ with a kernel $k$ of type ${\cal PC}$. 
By means of this result and Schauder's fixed point theorem we then prove existence in the 
non-degenerate case in Section \ref{SecND}. Section \ref{boundednessSec} provides global $L_\infty$-bounds
for the weak solutions in the non-degenerate case by the aid of De Giorgi's iteration technique. 
Using the previous results our main existence result in the degenerate case is then established in Section \ref{BddWS}. 
Finally, the $L_1$-contraction estimate is proved in Section \ref{L1C}.
\section{Preliminaries} \label{Prelim}
We first collect some properties of kernels of type $\mathcal{PC}$. Let $(k,l)\in \mathcal{PC}$.
For $\gamma\ge 0$ define the kernels $s_\gamma, r_\gamma \in L_{1,loc}(\iR_+)$ via the scalar Volterra equations (cf.\ \cite{GLS})
\begin{align*}
s_\gamma(t)+\gamma(l\ast s_\gamma)(t) & = 1,\quad t>0,\\
r_\gamma(t)+\gamma(l\ast r_\gamma)(t) & = l(t),\quad t>0.
\end{align*}
Both $s_\gamma$ and $r_\gamma$ are nonnegative for all $\gamma\ge 0$. This is a consequence of the complete positivity of $l$ (see \cite{CN}, \cite{JanI}). Furthermore, $s_\gamma$ belongs to $H^1_{1,\,loc}(\iR_+)$ and is nonincreasing for all 
$\gamma\ge 0$.

For $\gamma>0$ let $h_\gamma\in L_{1,loc}(\iR_+)$ denote the resolvent kernel associated
with $\gamma l$ (cf.\ \cite{GLS}), that is,
\begin{equation} \label{hndef}
h_\gamma(t)+\gamma(h_\gamma\ast l)(t)=\gamma l(t),\quad t>0.
\end{equation}
Observe that $h_\gamma=\gamma r_\gamma=-\dot{s}_\gamma \in L_{1,\,loc}(\iR_+)$, in particular $h_\gamma$ is nonnegative. It is well-known that for any $f\in L_p((0,T))$, $1\le p<\infty$, there holds
$h_n\ast f\rightarrow f$ in $L_p((0,T))$ as $n\rightarrow \infty$, see e.g.\ \cite{Za}.

For $\gamma>0$ we set
\begin{equation} \label{kndef}
k_{\gamma}=k\ast h_\gamma.
\end{equation}
It is known (see e.g.\ \cite{Za}) that $k_\gamma=\gamma s_\gamma$, $\gamma>0$, and thus
the kernels $k_\gamma$ are also nonnegative and nonincreasing, and they belong to
$H^1_{1,\,loc}(\iR_+)$ as well. The kernels $k_\gamma$ can be viewed as a regularized approximation of the kernel $k$
in the following sense. Letting $X$ be a real Banach space, the operator $B$ defined by
\[ B u=\partial_t(k\ast u),\;\;D(B)=\{u\in L_p((0,T);X):\,k\ast u\in \mbox{}_0 H^1_p((0,T);X)\},
\]
where the zero means vanishing at $t=0$, is known to be
$m$-accretive in $L_p((0,T);X)$, cf.\
\cite{Grip1}. Its Yosida approximations $B_{n}$, defined by
$B_{n}=nB(n+B)^{-1},\,n\in \iN$, have the property that for any
$u\in D(B)$, one has $B_{n}u\rightarrow Bu$ in $L_p((0,T);X)$ as
$n\to \infty$. Moreover, one has the representation (see \cite{VZ1})
\begin{equation} \label{Yos}
B_n u=\partial_t(k_n\ast u),\quad u\in L_p((0,T);X),\;n\in
\iN,
\end{equation}
thus the kernels $k_n$ take on the role of $k$ when replacing the operator $B$ by its Yosida approximations $B_n$.

We next state an important convexity inequality for operators of the form
$\partial_t(k\ast \cdot)$. A proof can be found in \cite{KSVZ}.
\begin{lemma} \label{FIconvex} Let $T>0$ and $U$ be an open subset of $\iR$. Let further $k\in
H^1_1((0,T))$ be nonnegative and nonincreasing, $H\in C^1(U)$ be convex, $u_0\in U$, and $u\in L_1((0,T))$ 
with $u(t)\in U$
for a.a. $t\in (0,T)$. Suppose that the functions $H(u)$, $H'(u)u$,
and $H'(u)(\dot{k}\ast u)$ belong to $L_1((0,T))$ (which is the case
if, e.g., $u\in L_\infty((0,T))$). Then
\begin{equation} \label{convexfundidentity}
H'(u(t))\partial_t\big(k \ast [u-u_0]\big)(t) \ge \;\partial_t\,\big(k\ast
[H(u)-H(u_0)]\big)(t),\quad \mbox{a.a.}\;t\in (0,T).
\end{equation}
\end{lemma}
We will apply this inequality several times with the kernels $k_\gamma$, $\gamma>0$, which are admissible.
\section{Compactness} \label{CompactnessSec}
In this section we establish a compactness criterion of Aubin-Lions type which is tailor-made for integro-differential
equations involving an operator $\partial_t(k\ast \cdot)$ with a kernel $k$ of type $\mathcal{PC}$.

Letting $f$ be a function on some interval $J\subset \iR$ we define $(\tau_{h}f)(t)=f(t+h)$ for $h>0$. 
If $J=[0,T]$, the translated function $\tau_{h}f$ is then defined on $[-h,T-h]$. 

The proof of our compactness criterion is based on the following general result.
\begin{satz} \cite[Theorem 5]{Si86} \label{SimonMain2}
Let $X, Y, Z$ be Banach spaces with continuous embeddings $X\hookrightarrow Y\hookrightarrow Z$ and assume in addition that the embedding $X \hookrightarrow Y$ is compact. Let $1\leq p < \infty$, $\mathcal F\subset L_p((0,T);Y)$ and assume that
\begin{itemize}
\item[(i)] $\mathcal F$ is bounded in $L_{p}((0,T);X)$,
\item[(ii)] $|\tau_{h}f-f|_{L_{p}((0,T-h);Z)} \rightarrow 0$ as $h \rightarrow 0$, uniformly for $f \in \mathcal F$.
\end{itemize}                                 
Then $\mathcal F$ is relatively compact in $L_{p}((0,T);Y)$.
\end{satz}
We come now to the main result of this section. 
\begin{satz} \label{comWVH}
Let $V$ and $H$ be real and separable Hilbert spaces such that $V$ is densely and compactly embedded
into $H$ and identify $H$ with its dual $H'$ so that $V \hookrightarrow H \hookrightarrow V'$. 
Let $1\le p<\infty$, $T>0$, $x\in H$ and 
$k\in L_{1,\,loc}(\iR_{+})$ be of type $\mathcal{PC}$. Let further
\begin{align*}
W(x,V, H)=\{u\in L_{p}((0,T);V)\,:\,k\ast (u-x)\in {}_0H^{1}_{p} ((0,T);V')\}.
\end{align*}
Then for any $m>0$ the set
\[
W_m:=\big\{u\in W(x,V, H)\,:\,|u|_{ L_{p}((0,T);V)}+|k\ast (u-x)|_{H^{1}_{p} ((0,T);V')}\le m\big\}
\]
is relatively compact in $L_p((0,T);H)$.
\end{satz}
\begin{proof}
We will show the relative compactness of $W_m$ in $L_p((0,T);H)$ by means of Theorem \ref{SimonMain2}
with $X=V$, $Y=H$ and $Z=V'$. Evidently, $W_m$ is bounded in $L_{p}((0,T);V)$, by definition, and thus
$W_m$ satisfies condition (i) in Theorem \ref{SimonMain2}. It remains to show that
\begin{equation} \label{cond2}
|\tau_{h}u-u|_{L_{p}((0,T-h); V')} \rightarrow 0\,\,\, \mbox{as}\,\, h \rightarrow 0\,\,\mbox{uniformly for}\, 
u\in W_m.
\end{equation}

Let $u\in W_m$. Then $k\ast (u-x)\in { }_0 H^{1}_{p} ((0,T);V')$ and thus $v:=\partial_t(k\ast [u-x])$
belongs to $L_p((0,T);V')$. Since $k\ast (u-x)(0)=0$ and $k\ast l=1$ (with $l$ being the kernel from 
condition ($\mathcal{PC}$)) we have $u=l\ast v+x$. In fact,
\begin{align*}
u-x & = \partial_t\big(l\ast k\ast [u-x]\big)
 =l\ast \partial_t(k\ast [u-x])= l\ast v.
\end{align*}
Using this expression for $u$ we have
\begin{align*}
|\tau_h u-u& |_{L_{p}((0,T-h); V')}^p = \int_0^{T-h}|(l\ast v)(t+h)-(l\ast v)(t)|_{V'}^p\,dt\\
& = \int_0^{T-h} \big|\int_0^{t+h} l(t+h-s)v(s)\,ds-\int_0^t l(t-s)v(s)\,ds\big|_{V'}^p\,dt\\
& =  \int_0^{T-h} \big|\int_0^t \big(l(t+h-s)-l(t-s)\big)v(s)\,ds+\int_t^{t+h}l(t+h-s)v(s)\,ds\big|_{V'}^p
\,dt\\
& \le 2^p\big(I_1(v,h)+I_2(v,h)\big),
\end{align*}
where
\begin{align*}
I_1(v,h) &= \int_0^{T-h} \big|\int_0^t \big(l(t+h-s)-l(t-s)\big)v(s)\,ds\big|_{V'}^p\,dt,\\
I_2(v,h) &= \int_0^{T-h} \big|\int_t^{t+h}l(t+h-s)v(s)\,ds\big|_{V'}^p\,dt.
\end{align*}
Setting $\tilde{l}_h=\tau_h l-l$ and using Young's inequality for convolutions, we may estimate $I_1(v,h)$ as follows:
\begin{align*}
I_1(v,h) & =\int_0^{T-h} \big| (\tilde{l}_h\ast v)(t)\big|^p_{V'}\,dt
=|\tilde{l}_h\ast v|_{L_p((0,T-h);V')}^p\\
& \le |\tilde{l}_h|^p_{L_1((0,T-h))}|v|_{L_p((0,T-h);V')}^p\\
& = |\tilde{l}_h|^p_{L_1((0,T-h))} |\partial_t(k\ast [u-x])|_{L_p((0,T-h);V')}^p\\
& \le m^p |\tilde{l}_h|^p_{L_1((0,T))}. 
\end{align*}
By continuity of translation on $L_1$, we have $|\tilde{l}_h|_{L_1((0,T))}\to 0$ as $h\to 0$ and therefore
$I_1(v,h)\to 0$ as $h\to 0$, uniformly w.r.t.\ $v=\partial_t(k\ast [u-x])$ with $u\in W_m$. 

As to $I_2(v,h)$, we first substitute $\sigma=t+h-s$ in the inner integral and then $\tilde{t}=t+h$ in the outer integral
to obtain
\begin{align*}
I_2(v,h) & =\int_0^{T-h} \big|\int_0^{h}l(\sigma)v(t+h-\sigma)\,d\sigma\big|_{V'}^p\,dt  \\
 & = \int_h^{T} \big|\int_0^{h}l(\sigma)v(\tilde{t}-\sigma)\,d\sigma\big|_{V'}^p\,d\tilde{t}\\
 & = |[\chi_{(0,h)}l]\ast v|_{L_p((h,T);V')}^p, 
\end{align*}
where $\chi_{(0,h)}$ denotes the characteristic function of the interval $(0,h)$. Applying Young's inequality for convolutions
it follows that
\begin{align*}
I_2(v,h) & \le  |[\chi_{(0,h)}l]\ast v|_{L_p((0,T);V')}^p\le  |[\chi_{(0,h)}l]|_{L_1((0,T))}^p |v|_{L_p((0,T);V')}^p\\
 & \le m^p  |[\chi_{(0,h)}l]|_{L_1((0,T))}^p.
\end{align*}
By the dominated convergence theorem, $|[\chi_{(0,h)}l]|_{L_1((0,T))}\to 0$ as $h\to 0$ and thus  
$I_2(v,h)\to 0$ as $h\to 0$, uniformly w.r.t.\ $v=\partial_t(k\ast [u-x])$ where $u\in W_m$. 

From the preceding estimates we conclude that \eqref{cond2} holds true and thus the proof of the theorem is complete.
\end{proof}

The following compactness result is similar to Theorem \ref{comWVH}, however only two Hilbert spaces $V$ and $H$
are involved, in particular no Gelfand triple. We will apply this result in Section \ref{BddWS} (with $V=L_2(\Omega)$ and 
$H=H^{-1}(\Omega)$) in the final step of the proof of the main result of this paper.
But, like Theorem \ref{comWVH}, it is also of independent interest.
\begin{satz} \label{comWVH2}
Let $V$ and $H$ be real and separable Hilbert spaces such that $V$ is densely and compactly embedded
into $H$. 
Let $1\le p<\infty$, $T>0$, $x\in V$ and 
$k\in L_{1,\,loc}(\iR_{+})$ be of type $\mathcal{PC}$. Let further
\begin{align*}
\tilde{W}(x,V, H)=\{u\in L_{p}((0,T);V)\,:\,k\ast (u-x)\in {}_0H^{1}_{p} ((0,T);H)\}.
\end{align*}
Then for any $m>0$ the set
\[
\tilde{W}_m:=\big\{u\in \tilde{W}(x,V, H)\,:\,|u|_{ L_{p}((0,T);V)}+|k\ast (u-x)|_{H^{1}_{p} ((0,T);H)}\le m\big\}
\]
is relatively compact in $L_p((0,T);H)$.
\end{satz}
\begin{proof}
The proof is analogous to the one of Theorem \ref{comWVH}. To obtain the relative compactness of 
$\tilde{W}_m$ in $L_p((0,T);H)$,
Theorem \ref{SimonMain2} is now applied with $X=V$ and $Y=Z=H$. The line of arguments is then exactly
the same as before.
\end{proof}
\begin{bemerk1}
{\em
We remark that in the situation of a Gelfand triple
${V} \hookrightarrow {H} \hookrightarrow {V}'$ where the Hilbert space ${V}$
is densely and compactly embedded into the Hilbert space ${H}$, then also ${H}$ is densely and
compactly embedded into ${V'}$ and thus one can apply Theorem \ref{comWVH2} to both of the pairs
$(V,H)$ and $(H,V')$.
}
\end{bemerk1}
\section{Existence in the non-degenerate case} \label{SecND}
In this section we consider the problem \eqref{mainproblem} in the non-degenerate case, which means that $\varphi'$ is bounded away from zero on all of 
$\iR$. We will prove existence of a weak solution of this quasilinear problem by means of Schauder's fixed point theorem.
The basic idea of the proof is not new; it has already been applied to various types of problems, in particular in the
stationary case and to quasilinear parabolic equations with the usual time derivative. The main ingredients in our argument
are the Aubin-Lions type result from the previous section, Theorem \ref{comWVH}, and the subsequent existence result,
Theorem \ref{linearsol}, on 
the linear problem 
\begin{equation} \label{linearhilfsproblem}
\left\{
\begin{array}{r@{\;=\;}l@{\;}l}
\partial_t\big(k\ast[u-u_0]\big)-\mbox{div}\big(A(t,x)\nabla u\big) & f,\quad & (t,x)\in \Omega_T\\
u & 0,\quad & (t,x)\in \Gamma_T \\
u|_{t=0} & u_0,\; & x\in \Omega.
\end{array}
\right.
\end{equation}
Here we use the notation $\Omega_T=(0,T)\times \Omega$, $\Gamma=\partial \Omega$, and $\Gamma_T=(0,T)\times \Gamma$.
The coefficients and data are supposed to satisfy the following assumptions.
\begin{itemize}
\item [{\bf (Hd)}] $\; u_0\in L_2(\Omega)$ and $f\in L_2((0,T);L_2(\Omega))$. 
\item [{\bf (HA)}] $\;A\in L_\infty((0,T)\times \Omega;\iR^{d\times
d})$, and there exists a $\nu>0$ such that
\[
(A(t,x)\xi|\xi)\ge \nu|\xi|^2,\quad \mbox{for a.a.}\;(t,x)\in
\Omega_T,\,\mbox{and all}\,\xi\in \iR^d.
\]
\end{itemize}
We set $\oH^1_2(\Omega):=\overline{C_0^\infty(\Omega)}\,{}^{H^1_2(\Omega)}$ and
denote by $H^{-1}(\Omega)$ the dual space of $\oH^1_2(\Omega)$. 
We say that $u$ is a {\em weak solution} of \eqref{linearhilfsproblem} if 
\[
u\in W(T,u_0):=\{w\in L_2((0,T);\oH^1_2(\Omega)):\;k\ast (w-u_0)\in {}_0H^{1}_{2} ((0,T);H^{-1}(\Omega))\}
\]
and for any test function
\[
\eta\in \oH^{1,1}_2(\Omega_T):= H^1_2((0,T);L_2(\Omega))\cap
L_2((0,T);\oH^1_2(\Omega))
\]
with $\eta|_{t=T}=0$ there holds
\[
\int_{0}^{T} \int_\Omega \Big(-\eta_t \big(k\ast [u-u_0]\big)+
(A\nabla u|\nabla \eta)\Big)\,dx\,dt=\,\int_0^T \int_\Omega f\eta\,dx\,dt.
\]

\begin{satz} \label{linearsol}
Let $T>0$, $\Omega$ be a bounded domain in $\iR^d$, and $(k,l)\in \mathcal{PC}$. Suppose that the assumptions (Hd) and (HA) are satisfied. Then
the problem \eqref{linearhilfsproblem} possesses a unique weak solution $u$ in the class $W(T,u_0)$, and there
exists a constant $M_0>0$ depending only on $\nu, |A|_{L_\infty(\Omega_T)}, l$, $\Omega$, and $T$ such that
\begin{equation} \label{linearest}
|k\ast (u-u_0)| {}_{H^{1}_{2} ((0,T);H^{-1}(\Omega))}+|u|_{L_2((0,T);H^1_2(\Omega))}\le
M_0\big(|u_0|_{L_2(\Omega)}+|f|_{L_2(\Omega_T)}).
\end{equation}
\end{satz}
This result follows directly from Theorem 3.1 in \cite{ZWH} (and its proof) and Poincar\'e's inequality.

Concerning the nonlinear problem \eqref{mainproblem} we impose the following non-degeneracy condition on
the function $\varphi$.
\begin{itemize}
\item [{\bf (H$\varphi$)}] $\varphi\in C^1(\iR)$ and there exist $c_0, c_1>0$ such that
$c_0\le \varphi'(r)\le c_1,\;r\in \iR$. 
\end{itemize}
We say that $u\in W(T,u_0)$ is a {\em weak solution} of \eqref{mainproblem} if 
for any test function $\eta\in \oH^{1,1}_2(\Omega_T)$
with $\eta|_{t=T}=0$ there holds
\begin{equation} \label{weakformulation}
\int_{0}^{T} \int_\Omega \Big(-\eta_t \big(k\ast [u-u_0]\big)+
(A\nabla \varphi(u)|\nabla \eta)\Big)\,dx\,dt=\,\int_0^T \int_\Omega f\eta\,dx\,dt.
\end{equation}
Observe that $u\in W(T,u_0)$ and the condition (H$\varphi$) ensure that $\varphi(u)\in L_2((0,T); H^1_2(\Omega))$
with $\nabla \varphi(u)=\varphi'(u)\nabla u\in L_2(\Omega_T)$.

Our existence result in the non-degenerate case reads as follows.
\begin{satz} \label{existenceND}
Let $T>0$, $\Omega$ be a bounded domain in $\iR^d$, and $(k,l)\in \mathcal{PC}$. Suppose that the assumptions (Hd), (HA), and (H$\varphi$) are satisfied.
Then the problem \eqref{mainproblem} admits a weak solution $u$ in the class $W(T,u_0)$.
\end{satz}
\begin{proof}
Define the map $S:\,L_2(\Omega_T)\rightarrow L_2(\Omega_T)$ which assigns to any $u\in L_2(\Omega_T)$ the
unique weak solution $v=S(u)$ of the linear problem
\begin{equation} \label{lpfixedpointmap}
\left\{
\begin{array}{r@{\;=\;}l@{\;}l}
\partial_t\big(k\ast[v-u_0]\big)-\mbox{div}\big(A(t,x)\varphi'(u)\nabla v\big) & f,\quad & (t,x)\in \Omega_T\\
v & 0,\quad & (t,x)\in \Gamma_T \\
v|_{t=0} & u_0,\; & x\in \Omega.
\end{array}
\right.
\end{equation}
This map is well defined thanks to Theorem \ref{linearsol}. In fact, for any $u\in L_2(\Omega_T)$ the coefficient
matrix $A\varphi'(u)$ belongs to $L_\infty(\Omega_T;\iR^{d\times
d})$, and 
\[
\big(A(t,x)\varphi'(u(t,x))\xi|\xi\big)\ge c_0 \nu|\xi|^2,\quad \mbox{for a.a.}\;(t,x)\in
\Omega_T,\,\mbox{and all}\,\xi\in \iR^d.
\] 
From Theorem \ref{linearsol} we further know that the solution $v$ of \eqref{lpfixedpointmap} lies in the class
$W(T,u_0)$ satisfying an estimate
\begin{equation} \label{linearestS}
|k\ast (v-u_0)| {}_{H^{1}_{2} ((0,T);H^{-1}(\Omega))}+|v|_{L_2((0,T);H^1_2(\Omega))}\le M,
\end{equation} 
where the bound $M>0$ can be chosen in such a way that it only depends on $\nu$, $c_0$, $c_1$, 
$|A|_{L_\infty(\Omega_T)}$,
 $l$, $\Omega$, $T$, $|u_0|_{L_2(\Omega)}$ and $|f|_{L_2(\Omega_T)}$. Set $K=\{w\in W(T,u_0):$ $w$ satisfies
 estimate \eqref{linearestS}$\}$.
Denoting by $B_M$ the closed ball
 in $L_2(\Omega_T)$ with radius $M$ and center $0$, the estimate \eqref{linearestS} implies that
 $v=S(u)\in B_M$ for all $u\in L_2(\Omega_T)$, in particular $S$ leaves the non-empty, closed, bounded, and convex set
 $B_M$ invariant. Furthermore, it follows from estimate \eqref{linearestS} and Theorem \ref{comWVH} with $p=2$,
 $V=\oH^1_2(\Omega)$, and $H=L_2(\Omega)$ that
 $S$ maps $L_2(\Omega_T)$ into a set (namely $K$) which is relatively compact in $L_2(\Omega_T)$.
 
 Next, we show that $S$ is continuous. Let $u_n\to u$ in $L_2(\Omega_T)$ and set $v_n=S(u_n)$. In order to show that
 $v_n\to S(u)$ in $L_2(\Omega_T)$ we prove that each subsequence of $(v_n)$ has a subsequence converging to
 $S(u)$ in $L_2(\Omega_T)$. Thus we may deliberately take subsequences in what follows. Denoting an arbitrarily fixed
 subsequence of $(v_n)$ again by $(v_n)$ we may assume (after taking a subsequence) that $u_n\to u$ a.e.\ in $\Omega_T$,
 which in turn implies that $\varphi'(u_n)\to \varphi'(u)$ a.e.\ in $\Omega_T$.
 Furthermore, since $(v_n)$ lies $K$, which is relatively compact in $L_2(\Omega_T)$ and bounded in the reflexive space $L_2((0,T);\oH^1_2(\Omega))$, we may assume (after taking
 a subsequence) that $v_n\to v$  in $L_2(\Omega_T)$ and $\nabla v_n \rightharpoonup \nabla v$  in $L_2(\Omega_T)$
 for some $v\in L_2((0,T);\oH^1_2(\Omega))$. Finally, $v_n=S(u_n)$ is equivalent to
\begin{equation} \label{Scont}
\int_{0}^{T} \int_\Omega \Big(-\eta_t \big(k\ast [v_n-u_0]\big)+
(A\varphi'(u_n)\nabla v_n| \nabla \eta)\Big)\,dx\,dt=\,\int_0^T \int_\Omega f\eta\,dx\,dt
\end{equation}
for any test function $\eta\in \oH^{1,1}_2(\Omega_T)$ with $\eta|_{t=T}=0$. Writing $A(t,x)=(a_{ij}(t,x))_{i,j=1,\ldots,d}$,
the dominated convergence theorem gives the convergence
\[
a_{ij}(t,x)\varphi'(u_n)\partial_{x_i}\eta \to a_{ij}(t,x)\varphi'(u)\partial_{x_i}\eta
\]
in $L_2(\Omega_T)$ for all $i,j=1,\ldots,d$, where we use the boundedness of $A$ and $\varphi'$. Combining this
with the weak convergence $\nabla v_n \rightharpoonup \nabla v$  in $L_2(\Omega_T)$ we see that
\[
\int_{0}^{T} \int_\Omega 
(A\varphi'(u_n)\nabla v_n| \nabla \eta)\,dx\,dt \to \int_{0}^{T} \int_\Omega 
(A\varphi'(u)\nabla v| \nabla \eta)\,dx\,dt\quad \mbox{as}\;n\to \infty.
\]
Since $u_0\in L_2(\Omega)$ and $v_n\to v$  in $L_2(\Omega_T)$, we also have that $k\ast (v_n-u_0)\to k\ast (v-u_0)$,
by Young's inequality for convolutions. Hence, sending $n\to \infty$ in \eqref{Scont} yields
\[
\int_{0}^{T} \int_\Omega \Big(-\eta_t \big(k\ast [v-u_0]\big)+
(A\varphi'(u)\nabla v| \nabla \eta)\Big)\,dx\,dt=\,\int_0^T \int_\Omega f\eta\,dx\,dt
 \] 
for any test function $\eta\in \oH^{1,1}_2(\Omega_T)$ with $\eta|_{t=T}=0$, which is equivalent to $v=S(u)$.
Hence $S$ is continuous so that we may apply Schauder's fixed point theorem to obtain existence of a $u\in W(T,u_0)$
such that $S(u)=u$. This $u$ is a weak solution of \eqref{mainproblem}.
\end{proof}
\section{$L_\infty$-bounds for weak solutions in the non-degenerate case} \label{boundednessSec}
The aim of this section is to derive $L_\infty$-bounds for weak solutions of the nonlinear problem 
\eqref{mainproblem} in the non-degenerate case, which, concerning lower bounds of $\varphi'$, do not depend
on the constant $c_0$ from (H$\varphi$) but only on a lower bound for $\varphi'(r)$ for sufficiently large $|r|$.
To achieve our goal we also have to strengthen the assumptions on the data and the kernel $k$.
We will now assume that $(k,l)\in \mathcal{PC}$ where the kernel $l\in L_p((0,T))$ for some $p>1$,
which is the case for $(k,l)=(g_{1-\alpha},g_\alpha)$ with $\alpha\in(0,1)$ and the pairs $(k,l)$ described in
Section \ref{IntroSec} (i),(ii). Fixing such a
$p$ we further assume that
\begin{itemize}
\item [{\bf (H'd)}] $\; u_0\in L_\infty(\Omega)$ and $f\in L_{q_1}((0,T);L_{q_2}(\Omega))$. 
where
\[
\frac{p'}{q_1}\,+\,\frac{d}{2q_2}\,=1-\beta,
\]
and
\[
\begin{array}{l@{\quad \mbox{for}\quad}l}
q_1\in \Big[\,\frac{p'}{(1-\beta)}\,,\infty\Big],\;
q_2\in \Big[\,\frac{d}{2(1-\beta)}\,,\infty\Big],\;\beta\in (0,1) & d\ge 2,\\
q_1\in
\Big[\,\frac{p'}{(1-\beta)}\,,\frac{2p'}{(1-2\beta)}\,\Big],\;
q_2\in [1,\infty],\;\beta\in \Big(0,\frac{1}{2}\Big) & d=1.
\end{array}
\]
\end{itemize}
The boundedness result reads as follows.
\begin{satz} \label{bddnessND}
Let $T>0$, $\Omega$ be a bounded domain in $\iR^d$, and $(k,l)\in \mathcal{PC}$ where the kernel $l\in L_p((0,T))$ for some $p>1$. Suppose that the assumptions (H'd), (HA), and (H$\varphi$) are satisfied. Let further $R>0$ and assume
in addition that
\begin{equation} \label{varphimu}
0<\mu\le \varphi'(r)\quad \mbox{for all}\;r\in \iR\;\mbox{with}\;|r|\ge R.
\end{equation}
Then any weak solution $u\in W(T,u_0)$ of problem \eqref{mainproblem} is essentially bounded in $\Omega_T$ and
\begin{equation}
|u|_{L_\infty(\Omega_T)}\le C\big(1+\max\{R,|u_0|_{L_\infty(\Omega)}\}\big)
\end{equation}
where $C=C(p,q_1,q_2,|l|_{L_p((0,T))},T,d,\mu,\Omega,|f|_{L_{q_1}((0,T);L_{q_2}(\Omega))})\ge 1$, in particular $C$ is 
independent of the constants $c_0$ and $c_1$ from (H$\varphi$).
\end{satz}
\begin{proof}
Theorem \ref{bddnessND} can be proved by means of a straightforward adaption of the arguments given in \cite{Za}, 
see in particular the proof of 
\cite[Theorem 3.1]{Za}. First of all, by introducing the Yosida approximation
of the operator $\partial_t(k\ast \cdot)$, one can obtain an equivalent weak formulation for \eqref{mainproblem}
where the involved kernel is more regular than $k$, see \cite[Lemma 3.1]{Za} and Section \ref{BddWS} below. This time-regularization of the problem
allows to take the test functions $u_\kappa^+:=(u-\kappa)_+$, where $y_+$ stands for the positive part of $y\in \iR$ and $\kappa\ge \kappa_0:=\max\{R,|u_0|_{L_\infty(\Omega)}\}$. Proceeding as in the proof of \cite[Theorem 3.1]{Za}
one obtains the estimate 
\begin{equation} \label{GB1}
\frac{1}{2}\,\int_{\Omega} (u_\kappa^+)^2\,dx+l\ast \int_{\Omega}
(A\varphi'(u)\nabla u|\nabla u_\kappa^+)\,dx\le l\ast \int_{\Omega}f u_\kappa^+\,dx\quad\mbox{a.e. in}\; (0,T),
\end{equation}
c.f.\ \cite[formula (23)]{Za}. Using (HA) and \eqref{varphimu} the second term on the left-hand side of \eqref{GB1} can
be estimated as follows.
\begin{align*}
\big(l\ast \int_{\Omega} &
(A\varphi'(u) \nabla u|\nabla u_\kappa^+)\,dx\big)(t)\\
& =\int_0^t l(t-\tau) \int_{\{u(\tau,\cdot)>\kappa\}}
\big(A(\tau,x)\varphi'(u(\tau,x))\nabla u(\tau,x)\big|\nabla u(\tau,x)\big)\,dx\,d\tau\\
& \ge \nu\mu \int_0^t l(t-\tau) \int_{\{u(\tau,\cdot)>\kappa\}}\big|\nabla u(\tau,x)\big|^2\,dx\,d\tau\\
& = \nu \mu \big(l\ast \int_\Omega |\nabla u_\kappa^+|^2\,dx\big)(t),\quad \mbox{a.a.}\;t\in (0,T).
\end{align*}
Following then the line of arguments given in the proof of \cite[Theorem 3.1]{Za} after formula (24) one 
first derives the basic truncated energy estimate, which in turn, by means of De Giorgi's iteration technique, leads
to an upper bound for $u$ of the asserted form. The corresponding lower bound for $u$ can be obtained by testing
the problem with $(-u-\kappa)_+$, $\kappa\ge \kappa_0$, and arguing as before now with $-u$ in place of $u$.
\end{proof}
\section{Bounded weak solutions in the degenerate case} \label{BddWS}
We turn now to the degenerate case. As in the previous section we assume that
\begin{itemize}
\item [{\bf (Hk)}] $(k,l)\in \mathcal{PC}$ where the kernel $l\in L_p((0,T))$ for some $p>1$.
\end{itemize}
We further assume (H'd), (HA) and that the nonlinearity $\varphi$ satisfies the following conditions.
\begin{itemize}
\item [{\bf (H'$\varphi$)}] $\varphi\in C^1(\iR)$, $\varphi'(r)\ge 0$ for all $r\in \iR$, $\varphi(0)=0$, $\varphi$ is strictly increasing in a neighbourhood of $0$, and there exist $\mu, R>0$ such that
\[
0<\mu\le \varphi'(r)\quad \mbox{for all}\;r\in \iR\;\mbox{with}\;|r|\ge R.
\]
\end{itemize}
We define
\begin{equation} \label{primi}
\Phi(r)=\int_0^r \varphi(s)\,ds,\quad r\in \iR.
\end{equation}

Setting 
\begin{align*}
W_{\varphi}(T,u_0):=\{w\in L_2((0,T);L_2(\Omega)): &\;k\ast (w-u_0)\in {}_0H^{1}_{2} ((0,T);H^{-1}(\Omega))\;
\mbox{and}\\
& \;\varphi(w)\in L_2((0,T);\oH^1_2(\Omega))\},
\end{align*}
we say that $u\in W_{\varphi}(T,u_0)$ is a weak solution of the degenerate problem \eqref{mainproblem} if
the relation \eqref{weakformulation} holds true for any $\eta\in \oH^{1,1}_2(\Omega_T)$ 
with $\eta|_{t=T}=0$.

Our main result on bounded weak solutions is the following.
\begin{satz} \label{degensol}
Let $T>0$, $\Omega$ be a bounded domain in $\iR^d$, and suppose that the assumptions (Hk), (H'd), (HA) 
and (H'$\varphi$) are satisfied. Then 
the problem \eqref{mainproblem} possesses a weak solution $u\in W_{\varphi}(T,u_0)\cap L_\infty(\Omega_T)$ and
\begin{align}
|\partial_t [k\ast (u-u_0)]| {}_{L_{2} ((0,T);H^{-1}(\Omega))}+& |\varphi(u)|_{L_2((0,T);H^1_2(\Omega))} \nonumber\\
& \le M_1\big(|u_0|_{L_2(\Omega)}+|\Phi(u_0)|_{L_1(\Omega)}+|f|_{L_2(\Omega_T)}\big). \label{nonlinearest}
\end{align}
where the constant $M_1$ only depends on $\nu, d, |A|_{L_\infty(\Omega_T)}, |k|_{L_1((0,T))}, \Omega, T$. 
Moreover 
\begin{equation} \label{nonlinearbound}
|u|_{L_\infty(\Omega_T)}\le C\big(1+\max\{R,|u_0|_{L_\infty(\Omega)}\}\big)
\end{equation}
where $C=C(p,q_1,q_2,|l|_{L_p((0,T))},T,d,\mu,\Omega,|f|_{L_{q_1}((0,T);L_{q_2}(\Omega))})\ge 1$ is the same
constant as in Theorem \ref{bddnessND}.
\end{satz}
\begin{proof}
Let 
\[
M=C\big(1+\max\{R,|u_0|_{L_\infty(\Omega)}\}\big)\quad (>R)
\]
be the $L_\infty(\Omega_T)$-bound provided in 
Theorem \ref{bddnessND} for arbitrary weak solutions of \eqref{mainproblem} under the assumptions 
(Hk), (H'd), (HA), (H$\varphi$) and \eqref{varphimu}. Recall that $M$ does not depend on the constants $c_0$ and
$c_1$ appearing in (H$\varphi$). In order to be able to apply our results from the non-degenerate case we first define
\begin{equation} \label{Mphi}
\tilde{\varphi}(r)=\left\{
\begin{array}{l@{\;:\;}l}
\varphi(r) & r\in [-M,M], \\
\varphi(M)+\varphi'(M)(r-M) & r>M,\\
\varphi(-M)+\varphi'(-M)(r+M) & r<-M,
\end{array}
\right.
\end{equation}
and then perturb $\tilde{\varphi}$ by setting
\[
\tilde{\varphi}_n(r)=\tilde{\varphi}(r)+\frac{1}{n}\,r,\quad r\in \iR,\,n\in \iN.
\] 
Evidently, $\tilde{\varphi}_n\in C^1(\iR)$,
\[
\frac{1}{n}\le \tilde{\varphi}_n'(r)\le |\varphi'|_{L_\infty((-M,M))}+1,\quad r\in \iR,\,n\in \iN, 
\]
and
\begin{equation} \label{tildebound}
0<\mu\le\tilde{\varphi}_n'(r) \quad \mbox{for all}\;n\in \iN\;\mbox{and}\; r\in \iR\;\mbox{with}\;|r|\ge R.
\end{equation}

Invoking Theorem \ref{existenceND}, we see that for every $n\in \iN$ the perturbed (non-degenerate) problem
\begin{equation} \label{perturbedproblem}
\left\{
\begin{array}{r@{\;=\;}l@{\;}l}
\partial_t\big(k\ast[u-u_0]\big)-\mbox{div}\big(A(t,x)\nabla \tilde{\varphi}_n(u)\big) & f,\quad & (t,x)\in\Omega_T\\
u & 0,\quad & (t,x)\in \Gamma_T\\
u|_{t=0} & u_0,\; & x\in \Omega,
\end{array}
\right.
\end{equation}
possesses a weak solution $u_n\in W(T,u_0)$. Since $\tilde{\varphi}_n$ satisfies the assumptions of
Theorem \ref{bddnessND} with uniform bound \eqref{tildebound}, this theorem yields the uniform estimate
\begin{equation} \label{unbound}
|u_n|_{L_\infty(\Omega_T)}\le M,\quad n\in \iN.
\end{equation}
Note that as $\tilde{\varphi}$ and $\varphi$ coincide on $[-M,M]$, \eqref{unbound} implies that
\[
\tilde{\varphi}_n(u_n)={\varphi}_n(u_n)
\]
where we set
\[
{\varphi}_n(r)=\varphi(r)+\frac{1}{n}\,r,\quad r\in \iR,\,n\in \iN.
\]

Next, we want to derive a uniform estimate for $\nabla {\varphi}_n(u_n)$. To this end we use the time-regularized
version of the weak formulation of \eqref{perturbedproblem}, which says that $u_n \in W(T,u_0)$ solves 
\eqref{perturbedproblem} in the weak sense if and only if
\begin{equation} \label{regWF}
\int_\Omega \Big(\psi \partial_t (k_m\ast [u_n-u_0])+\big(h_m\ast [A\nabla {\varphi}_n(u_n)]\big|
\nabla \psi\big)\Big)\,dx=\int_\Omega (h_m\ast f)\psi\,dx
\end{equation}
a.e.\ in $(0,T)$, for any test function $\psi\in \oH^1_2(\Omega)$ and all $m\in \iN$. Here $k_m$ and $h_m$ are the
functions introduced in Section \ref{Prelim}.
The equivalence of the two weak formulations follows directly from \cite[Theorem 3.1]{Za} and its proof, noting that
the problem can be regarded as a linear problem with coefficient matrix $A\tilde{\varphi}_n'(u_n)$.

For $t\in (0,T)$ we take in \eqref{regWF} the admissible test function $\psi={\varphi}_n(u_n)$; recall that
$\varphi_n(0)=0$ and $\varphi_n'(u_n)$ is bounded in $\Omega_T$. Setting $v_n={\varphi}_n(u_n)$ we obtain
\begin{equation} \label{Est1}
\int_\Omega \Big(\varphi_n(u_n) \partial_t (k_m\ast [u_n-u_0])+\big(h_m\ast [A\nabla v_n]\big|
\nabla v_n\big)\Big)\,dx=\int_\Omega (h_m\ast f)v_n\,dx
\end{equation}
a.e.\ in $(0,T)$. 

We define
\[
\Phi_n(r)=\int_0^r \varphi_n(s)\,ds,\quad r\in \iR,\,n\in \iN.
\]
Then $\Phi_n'=\varphi_n$ and thus $\Phi_n''=\varphi_n'=\varphi'+\frac{1}{n}\ge 0$ showing convexity of $\Phi_n$.
The last property allows us to apply the convexity inequality \eqref{convexfundidentity}, which gives
\begin{equation} \label{Est2}
\varphi_n(u_n) \partial_t \big(k_m\ast [u_n-u_0]\big)\ge \partial_t \big(k_m\ast [\Phi_n(u_n)-\Phi_n(u_0)]\big)
\quad \mbox{a.e. in}\;\Omega_T.
\end{equation}
From \eqref{Est1} and \eqref{Est2} we deduce that
\begin{equation} \label{Est3}
\int_\Omega \partial_t \big(k_m\ast [\Phi_n(u_n)-\Phi_n(u_0)]\big)\,dx+\int_\Omega \big(h_m\ast [A\nabla v_n]\big|
\nabla v_n\big)\,dx\le \int_\Omega (h_m\ast f)v_n\,dx
\end{equation}
a.e.\ in $(0,T)$. Observe that
\[
k_m\ast   [\Phi_n(u_n)-\Phi_n(u_0)]\in {}_0H_1^1((0,T);L_\infty(\Omega)).
\]
From this and $k_m=k\ast h_m$ as well as $k\ast l=1$ it follows that
\begin{align*}
l\ast \partial_t \big(k_m\ast [\Phi_n(u_n)-\Phi_n(u_0)]\big) & =
\partial_t \big(l\ast k_m \ast [\Phi_n(u_n)-\Phi_n(u_0)]\big)\\
& = h_m\ast  [\Phi_n(u_n)-\Phi_n(u_0)].
\end{align*} 
We convolve \eqref{Est3} with the nonnegative kernel $l$ and use the last identity; sending then $m\to \infty$ and choosing an
appropriate subsequence, if necessary, we find that
\begin{equation} \label{Est4} 
\int_\Omega \big(\Phi_n(u_n)-\Phi_n(u_0)\big)\,dx+l\ast \int_\Omega \big(A\nabla v_n\big|\nabla v_n\big)\,dx
\le l\ast \int_\Omega f v_n\,dx
\end{equation}
a.e.\ in $(0,T)$. 

Next, observe that $\Phi_n(u_n)\ge 0$ a.e.\ in $\Omega_T$. Thus we may drop this term in the first integral in \eqref{Est4}.
Estimating further the second integral by using the uniform parabolicity condition for $A$ and convolving the resulting
inequality with $k$ we obtain
\[
\nu (1\ast \int_\Omega |\nabla v_n|^2\,dx)(t) \le (1\ast k)(t) \int_\Omega \Phi_n(u_0)\,dx+
(1\ast \int_\Omega f v_n\,dx)(t) 
\]  
for all $t\in [0,T]$. Evaluating at $t=T$ and employing H\"older's inequality as well as Poincar\'e's inequality
in the form
\[
|w|_{L_2(\Omega)}\le C_P |\nabla w|_{L_2(\Omega;\iR^d)},\quad w\in \oH^1_2(\Omega),
\] 
we deduce further that
\[
\nu \int_0^T \int_\Omega |\nabla v_n|^2\,dx\,dt\le |k|_{L_1((0,T))}\int_\Omega \Phi_n(u_0)\,dx+
C_P|f|_{L_2(\Omega_T)}|\nabla v_n|_{L_2(\Omega_T;\iR^d)},
\]
which in turn yields
\begin{equation} \label{Est5}
\int_0^T \int_\Omega |\nabla v_n|^2\,dx\,dt\le \,\frac{2}{\nu}\,|k|_{L_1((0,T))}\int_\Omega \Phi_n(u_0)\,dx
+\frac{C_P^2}{\nu^2}\,|f|_{L_2(\Omega_T)}^2,
\end{equation}
by Young's inequality.

Recalling \eqref{primi} we have 
\[
\Phi_n(r)=\Phi(r)+\frac{1}{2n}\,r^2,
\]
and therefore
\[
\int_\Omega \Phi_n(u_0)\,dx\le |\Phi(u_0)|_{L_1(\Omega)}+\frac{1}{2}\,|u_0|_{L_2(\Omega)}^2.
\]
Combining this and \eqref{Est5} yields a uniform bound
\begin{equation} \label{gradbound}
|\nabla v_n|_{L_2(\Omega_T;\iR^d)}\le C_0,\quad n\in \iN,
\end{equation}
where the constant $C_0$ only depends on the data of the problem, but not on $n$.

By virtue of the uniform estimates \eqref{unbound} and \eqref{gradbound} there is a subsequence of $(u_n)$,
again denoted by $(u_n)$, such that as $n\to \infty$,
\begin{equation} \label{unschwach}
u_n \rightharpoonup u\quad \mbox{in}\;L_2((0,T);L_2(\Omega))
\end{equation}
and
\begin{equation} \label{vnschwach}
v_n =\varphi_n(u_n)=\tilde{\varphi}_n(u_n) \rightharpoonup v \quad \mbox{in}\;L_2((0,T);\oH^1_2(\Omega))
\end{equation}
for some $u\in L_2((0,T);L_2(\Omega))$ and $v\in L_2((0,T);\oH^1_2(\Omega))$. Sending $n\to \infty$
in the weak formulation of \eqref{perturbedproblem}, which reads
\begin{equation} \label{weakformun}
\int_{0}^{T} \int_\Omega \Big(-\eta_t \big(k\ast [u_n-u_0]\big)+
(A\nabla v_n| \nabla \eta)\Big)\,dx\,dt=\,\int_0^T \int_\Omega f\eta\,dx\,dt
\end{equation}
for any test function $\eta\in \oH^{1,1}_2(\Omega_T)$ with $\eta|_{t=T}=0$, we obtain
\begin{equation} \label{limitgl}
\int_{0}^{T} \int_\Omega \Big(-\eta_t \big(k\ast [u-u_0]\big)+
(A\nabla v| \nabla \eta)\Big)\,dx\,dt=\,\int_0^T \int_\Omega f\eta\,dx\,dt
\end{equation}
for all $\eta\in \oH^{1,1}_2(\Omega_T)$ with $\eta|_{t=T}=0$.

It remains to show that $v=\varphi(u)$. To see this, note first that $f\in L_2(\Omega_T)$, the estimate 
\eqref{gradbound} and the equation \eqref{weakformun} imply that the sequence $(k\ast [u_n-u_0])$ is 
bounded in ${}_0H^{1}_{2} ((0,T);H^{-1}(\Omega))$. Since we also know that $u_n$ is bounded in 
$L_2((0,T);L_2(\Omega))$, we may apply Theorem \ref{comWVH2} with 
$p=2$, $V=\oH^1_2(\Omega)$, and $H=L_2(\Omega)$, which ensures that there is a subsequence of $(u_n)$,
again denoted by $(u_n)$, which converges (strongly) in $L_2((0,T);H^{-1}(\Omega))$. Consequently,
\begin{equation} \label{unstrong}
u_n \to u\quad \mbox{in}\;L_2((0,T);H^{-1}(\Omega)).
\end{equation}

The desired relation $v=\varphi(u)$ can now be seen by means of a well known argument based on Minty's lemma,
see e.g.\ \cite{Schw} or \cite[Section II.2]{Show}. For the reader's convenience we provide some details. Define the operator 
$S:\,L_2(\Omega_T)\rightarrow L_2(\Omega_T)$ by $S(w)=\tilde{\varphi}(w)$. Observe that $S$ indeed maps
$L_2(\Omega_T)$ into itself, since $\tilde{\varphi}(r)$ grows linearly for large $|r|$. Using the (global) Lipschitz
continuity of $\tilde{\varphi}$ it is also easy to see that $S$ is continuous. Further, $S$ is monotone by monotonicity of $\tilde{\varphi}$. Employing the monotonicity of $\tilde{\varphi}_n$,
for any test function $w:\,\overline{\Omega_T}\rightarrow \iR$ we have
\begin{align*}
0 & \le \int_{\Omega_T}\big(v_n-\tilde{\varphi}_n(w)\big)\big(u_n-w\big)\,dx\,dt\\
 & = \int_{\Omega_T} \big(v_n u_n-\tilde{\varphi}_n(w) u_n-v_n w+\tilde{\varphi}_n(w) w\big)\,dx\,dt\\
 & \to \int_{\Omega_T} \big(v u-\tilde{\varphi}(w) u-v w+\tilde{\varphi}(w) w\big)\,dx\,dt
 = \int_{\Omega_T}\big(v-S(w)\big)\big(u-w\big)\,dx\,dt.
\end{align*}
In fact, the product term converges due to \eqref{vnschwach} and \eqref{unstrong}. For the other terms one uses \eqref{unschwach}
and \eqref{vnschwach} and the fact that $\tilde{\varphi}_n$ converges to $\tilde{\varphi}$ on any compact subset of $\iR$.
The obtained inequality for $S$ extends to arbitrary $w\in L_2(\Omega_T)$, by approximation. Minty's lemma then
shows that $v=S(u)=\tilde{\varphi}(u)$. Finally, \eqref{unbound} and \eqref{unschwach} imply that 
$u\in L_\infty(\Omega_T)$ with $|u|_{L_\infty(\Omega_T)}\le M$, whence $v=\varphi(u)$ as $\tilde{\varphi}=
\varphi$ on $[-M,M]$.
\end{proof}
\section{$L_1$-contraction estimate and uniqueness} \label{L1C}
In this section we show that an $L_1$-contraction estimate holds provided the solutions are sufficiently regular. 

Suppose that the assumptions (HA) and (Hd) are satisfied and that $u\in W_{\varphi}(T,u_0)$ is a weak solution of
\eqref{mainproblem}, which means that for any $\eta\in \oH^{1,1}_2(\Omega_T)$
with $\eta|_{t=T}=0$,
\begin{equation} \label{weakformulation2}
\int_{0}^{T} \int_\Omega \Big(-\eta_t \big(k\ast [u-u_0]\big)+
(A\nabla \varphi(u)|\nabla \eta)\Big)\,dx\,dt=\,\int_0^T \int_\Omega f\eta\,dx\,dt.
\end{equation}
Now suppose that $u$ enjoys the additional regularity property
\begin{equation} \label{addreg}
k\ast (u-u_0)\in {}_0H^1_1((0,T);L_1(\Omega)).
\end{equation}
Restricting to bounded test functions $\eta$ in the described class, we may integrate by parts w.r.t.\ time in the
first integral in \eqref{weakformulation2}, thereby obtaining that
\begin{equation} \label{weakformulation3}
\int_{0}^{T} \int_\Omega \Big(\eta \partial_t\big(k\ast [u-u_0]\big)+
(A\nabla \varphi(u)|\nabla \eta)\Big)\,dx\,dt=\,\int_0^T \int_\Omega f\eta\,dx\,dt.
\end{equation}
Observe that \eqref{weakformulation3} does not require any time regularity of $\eta$. Hence, by means of an approximation argument, one can show that under the additional condition \eqref{addreg}, the identity 
\eqref{weakformulation3} is satisfied for all $\eta\in L_2((0,T);\oH^1_2(\Omega))\cap L_\infty(\Omega_T)$.
Evidently, by means of a cut-off function argument, this is equivalent to
\begin{equation} \label{weakformulation4}
\int_{0}^{t_1} \int_\Omega \Big(\eta \partial_t\big(k\ast [u-u_0]\big)+
(A\nabla \varphi(u)|\nabla \eta)\Big)\,dx\,dt=\,\int_0^{t_1} \int_\Omega f\eta\,dx\,dt,\;\;t_1\in (0,T],
\end{equation}
for all $\eta\in L_2((0,T);\oH^1_2(\Omega))\cap L_\infty(\Omega_T)$.
Using this formulation of \eqref{mainproblem} one can also reduce the requirements on $u$, $u_0$ and $f$.
In fact, it is enough to assume that $u_0\in L_1(\Omega)$, $f\in L_1(\Omega_T)$, and $u$ belongs to the set
\begin{align*}
W_{\varphi,1}(T,u_0):=\{w\in L_1(\Omega_T):\,& \,k\ast (w-u_0)\in {}_0H^1_1((0,T);L_1(\Omega))\;\;
\mbox{and}\\
& \varphi(w)\in L_2((0,T);\oH^1_2(\Omega))\}.
\end{align*}
\begin{satz}
\label{L1ConTheorem}
Let $T>0$, $\Omega$ be a bounded domain in $\iR^d$, and $(k,l)\in \mathcal{PC}$. Suppose that
(HA) is satisfied and that $\varphi\in C^1(\iR)$ is strictly increasing in $\iR$. Let further
$u_{0,i}\in L_1(\Omega)$, $f_i\in L_1(\Omega_T)$, $i=1,2$, and assume that $u_i\in 
W_{\varphi,1}(T,u_{0,i})$ is a solution
of problem \eqref{mainproblem} with $u_0=u_{0,i}$ and $f=f_i$ in the sense that \eqref{weakformulation4}
holds true for all $\eta\in L_2((0,T);\oH^1_2(\Omega))\cap L_\infty(\Omega_T)$.
Then
\begin{equation} \label{ContractionEstimate}
|u_1-u_2|_{L_1(\Omega_T)}\le T |u_{0,1}-u_{0,2}|_{L_1(\Omega)}+|l|_{L_1((0,T))}.
\end{equation}
\end{satz}
\begin{proof}
Suppose that $u_j$ is a solution of the problem under study with data $u_{0,j}$ and $f_j$, $j=1,2$. That is,
\begin{equation} \label{mainproblemJ}
\left\{
\begin{array}{r@{\;=\;}l@{\;}l}
\partial_t\big(k\ast[u_j-u_{0,j}]\big)-\mbox{div}\big(A(t,x)\nabla \varphi(u_j)\big) & f_j,\quad & (t,x)\in \Omega_T\\
u_j & 0,\quad & (t,x)\in \Gamma_T\\
u_j|_{t=0} & u_{0,j},\; & x\in \Omega.
\end{array}
\right.
\end{equation}
We set $v_j=\varphi(u_j)$, $j=1,2$, and
\[
u=u_1-u_2,\quad u_0=u_{0,1}-u_{0,2},\quad v=v_1-v_2, \quad f=f_1-f_2.
\]
For $\varepsilon>0$ we introduce the $C^\infty(\iR)$-function
\[
H_\varepsilon(y)=\sqrt{y^2+\varepsilon^2}-\varepsilon,\quad y\in \iR.
\]
Note that
\[
H_\varepsilon'(y)=\frac{y}{\sqrt{y^2+\varepsilon^2}},\quad H_\varepsilon''(y)=\frac{\varepsilon^2}{(y^2+\varepsilon^2)^{3/2}},
\quad y\in \iR,
\]
in particular $H_\varepsilon$ is convex, and $H_\varepsilon'$ as well as $H_\varepsilon''$ are bounded.

Observe that $H_\varepsilon'(v)\in L_2((0,T);\oH^1_2(\Omega))\cap L_\infty(\Omega_T)$, thus we may use $\eta=H_\varepsilon'(v)$
as test function in the weak formulation of the problem for both $u_1$ and $u_2$ (in the sense described in Theorem 
\ref{L1ConTheorem}). Taking the difference of the resulting identities we obtain
\[
\int_0^{t_1} \int_\Omega H_\varepsilon'(v) \partial_t\big(k\ast[u-u_0]\big)\,dx\,dt+
\int_0^{t_1}\int_\Omega\big(A\nabla v|H_\epsilon''(v)\nabla v \big)\,dx,\,dt =
\int_0^{t_1}\int_\Omega f H_\varepsilon'(v)\,dx\,dt,
\]
for all $t_1\in (0,T]$.
Since $H_\varepsilon$ is convex and by (HA), the second term on the left-hand side is nonnegative. We further have
$|H_\varepsilon'(y)|\le 1$ for all $y\in \iR$ and thus
\[
\int_0^{t_1}\int_\Omega f H_\varepsilon'(v)\,dx\,dt \le \int_0^{t_1} \int_\Omega |f|\,dx\,dt,\quad t_1\in (0,T].
\]
Consequently,
\begin{equation} \label{Cont1}
\int_0^{t_1}\int_\Omega H_\varepsilon'(v) \partial_t\big(k\ast[u-u_0]\big)\,dx\,dt \le \int_0^{t_1} \int_\Omega |f|\,dx\,dt,
\quad t_1\in (0,T].
\end{equation}
We now write
\begin{align*}
\int_0^{t_1} \int_\Omega H_\varepsilon'(v) \partial_t\big(k\ast&\, [u-u_0]\big)\,dx\,dt 
 =\int_0^{t_1} \int_\Omega H_\varepsilon'(u) \partial_t\big(k_m\ast[u-u_0]\big)\,dx\,dt\\
 & + \int_0^{t_1} \int_\Omega H_\varepsilon'(u)
 \Big(\partial_t\big(k\ast[u-u_0]\big)- \partial_t\big(k_m\ast[u-u_0]\big)\Big)\,dx\,dt\\
&  +\int_0^{t_1}\int_\Omega \big(H_\varepsilon'(v)- H_\varepsilon'(u)\big) \partial_t\big(k\ast[u-u_0]\big)\,dx\,dt,
\end{align*}
where the kernel $k_m$ ($m\in \iN$) is given by \eqref{kndef}.
Using again convexity of $H_\varepsilon$ we may apply the convexity inequality \eqref{convexfundidentity} to get
\begin{align*}
\int_0^{t_1}\int_\Omega H_\varepsilon'(u) \partial_t\big(k_m\ast[u-u_0]\big)\,dx\,dt & \ge \int_0^{t_1} \int_\Omega 
\partial_t\big(k_m\ast\big[H_\varepsilon(u)-H_\varepsilon(u_0)\big]\big)\,dx\,dt\\
& = \int_\Omega 
k_m\ast\big[H_\varepsilon(u)-H_\varepsilon(u_0)\big](t_1,x)\,dx,\quad t_1\in (0,T].
\end{align*}
From \eqref{Cont1} and the previous considerations, together with $|H_\varepsilon'(y)|\le 1$ for all $y\in \iR$, we infer that
\begin{equation} \label{Contr3}
\big(k_m\ast \int_\Omega\big(H_\varepsilon(u)-H_\varepsilon(u_0)\big)\,dx\big)(t) \le 
\int_0^T \int_\Omega |f|\,dx\,d\tau+I_\varepsilon+\tilde{I}_m,\quad t\in (0,T],
\end{equation}
where
\[
I_\varepsilon=\int_0^T \int_\Omega \big|H_\varepsilon'(v)- H_\varepsilon'(u)\big|\, \big|\partial_{\tau}\big(k\ast[u-u_0]\big|\,dx\,d\tau
\]
and
\[
\tilde{I}_m=\int_0^{T} \int_\Omega 
 \Big|\partial_{\tau}\big(k\ast[u-u_0]\big)- \partial_{\tau}\big(k_m\ast[u-u_0]\big)\Big|\,dx\,d\tau.
\]

Recall that $k_m=k\ast h_m$ and that for any $g\in L_1((0,T))$, $h_m\ast g\to g$ in $L_1((0,T))$ as $m\to \infty$.
Since $g:=k\ast \int_\Omega (H_\varepsilon(u)-H_\varepsilon(u_0))\,dx\in  L_1((0,T))$, it follows that the left-hand side
in \eqref{Contr3} tends to $g$ in $L_1((0,T))$, and also pointwise a.e.\ in $(0,T)$ if we choose an appropriate subsequence if 
necessary. Further, we know that $k\ast (u-u_0)\in  {}_0H^1_1((0,T);L_1(\Omega))$, which means that $u-u_0$ lies in
the domain $D(B)$ of the $m$-accretive operator $B:=\partial_t(k\ast \cdot)$ in $L_1((0,T);L_1(\Omega))$ (see Section \ref{Prelim}). Hence $\tilde{I}_m\to 0$
as $m\to \infty$, since $\partial_t(k_m\ast \cdot)$ is the Yosida approximation of $B$ (cf.\ Section \ref{Prelim}).
Consequently we obtain
\begin{equation} \label{Contr4}
\big(k\ast \int_\Omega\big(H_\varepsilon(u)-H_\varepsilon(u_0)\big)\,dx\big)(t) \le 
\int_0^T \int_\Omega |f|\,dx\,d\tau+I_\varepsilon,\quad \mbox{a.e.}\;t\in (0,T).
\end{equation}
Convolving \eqref{Contr4} with the kernel $l$ and evaluating at $t=T$ gives
\begin{equation} \label{Contr5}
\int_0^T \int_\Omega H_\varepsilon(u)\,dx\,dt\le T \int_\Omega H_\varepsilon(u_0)\,dx+|l|_{L_1((0,T))} \int_0^T \int_\Omega |f|\,dx\,d\tau+|l|_{L_1((0,T))} I_\varepsilon.
\end{equation}

Finally, we send $\varepsilon\to 0$ in \eqref{Contr5}. Since $H_\varepsilon(y)\to |y|$ and $H_\varepsilon(y)\le |y|$ 
for all $y\in \iR$, we obtain by dominated convergence
that the first term in \eqref{Contr5} tends to $|u|_{L_1(\Omega_T)}$ as $\varepsilon\to 0$. Analogously,
\[
\int_\Omega H_\varepsilon(u_0)\,dx \to |u_0|_{L_1(\Omega)}.
\]
Concerning the last term, we use that $|H'_\varepsilon(y)|\le 1$ and $H'_\varepsilon(y)\to \mbox{sign}(y)$, $y\in \iR$, with sign$(0)=0$. Since $\partial_t\big(k\ast[u-u_0])$ is integrable on $\Omega_T$, it follows by dominated convergence that
\[
I_\varepsilon \to \int_0^T \int_\Omega \big|\mbox{sign}\big(\varphi(u_1)-\varphi(u_2)\big)-\mbox{sign}(u_1-u_2)\big|\, \big|\partial_t\big(k\ast[u-u_0]\big)\big|\,dx\,dt
=0,
\]
since $\varphi$ is strictly increasing. This yields the desired contraction estimate \eqref{ContractionEstimate}.
\end{proof}

$\mbox{}$
{\footnotesize

$\mbox{}$

\noindent {\bf Petra Wittbold}, Fakult\"at f\"ur Mathematik, Universit\"at Duisburg-Essen, Thea-Leymann-Str.\ 9,
45127 Essen, Germany,\\
 E-mail: {\em petra.wittbold@uni-due.de}

\medskip

\noindent {\bf Patryk Wolejko}, Institut f\"ur Angewandte Analysis, Universit\"at Ulm, 89069 Ulm, Germany, \\E-mail: 
{\em patrykwolejko@hotmail.com}

\medskip

\noindent {\bf Rico Zacher}, Institut f\"ur Angewandte Analysis, Universit\"at Ulm, 89069 Ulm, Germany, \\
E-mail: {\em rico.zacher@uni-ulm.de}

}

\end{document}